\documentclass{amsart}
\usepackage{amsfonts,amssymb,amsmath,amsthm}
\usepackage{url}
\usepackage{enumerate}

\newtheorem{thrm}{Theorem}[section]
\newtheorem{lem}[thrm]{Lemma}

\newtheorem{exam}[thrm]{Example}
\newtheorem{cor}[thrm]{Corollary}
\theoremstyle{definition}

\numberwithin{equation}{section}
\email{kkaygisiz@gop.edu.tr,ademsahin@gop.edu.tr}
\input{tcilatex}

\begin{document}
\author{Kenan KAYGISIZ and Adem SAHIN}
\address{ Department of Mathematics\\
Art and Science Faculty, Gaziosmanpasa University\\
60240, TOKAT, TURKEY}
\title{Generalized Lucas Numbers and Relations with Generalized Fibonacci
Numbers}

\begin{abstract}
In this paper, we present a new generalization of the Lucas numbers by
matrix representation using Genaralized Lucas Polynomials. We give some
properties of this new generalization and some relations between the
generalized order-$k$ Lucas numbers \ and generalized order-$k$ Fibonacci
numbers. In addition, we obtain Binet formula and combinatorial
representation for generalized order-$k$ Lucas numbers by using properties
of generalized Fibonacci numbers.
\end{abstract}

\maketitle

\section{Introduction}

There are various types of generalization of Fibonacci and Lucas numbers.
For example Er[1] defined the generalized order-$k$ Fibonacci numbers(GO$k$%
F), K\i l\i \c{c}[6] defined the generalized order-$k$ Pell numbers(GO$k$P)
and Ta\c{s}\c{c}{\i }[3] defined the generalized order-$k$ Lucas numbers (GO$%
k$L). MacHenry[7] defined Generalized Fibonacci and Lucas Polynomials and
MacHenry[8] defined matrices $A_{(k)}^{\infty }$ and $D_{(k)}^{\infty }$
depending on these polynomials.  $A_{(k)}^{\infty }$ is reduced to GO$k$F
when $t_{i}=1$ and $A_{(k)}^{\infty }$ is reduced to GO$k$P when $t_{1}=2$
and $t_{i}=1$ (for $2\leq i\leq k$). This analogy shows the importance of
the matrix $A_{(k)}^{\infty }$ and Generalized Fibonacci and Lucas
polynomials in generalizations. However, Lucas generalization of Ta\c{s}\c{c}%
{\i }[3] is not compatible with the matrix $A_{(k)}^{\infty }$ and
Generalization Fibonacci and Lucas polynomials, we studied on generalized
order-$k$ Lucas numbers $l_{k,n}$ (GO$k$L) and $k$ sequences of the
generalized order-$k$ Lucas numbers $l_{k,n}^{i}$ ($k$SO$k$L) with the help
of Lucas Polynomials $G_{k,n}$ and the matrix $D_{(k)}^{\infty }$. In this
paper, after presenting a matrix representation of $l_{k,n}^{i}$, we derived
a relations between generalized order-$k$ Fibonacci numbers(GO$k$F) and GO$k$%
L, as well as relation between $k$SO$k$L and $k$ sequences of the
generalized order-$k$ Fibonacci numbers $f_{k,n}^{\text{ }i}$($k$SO$k$F)$.$
Since many properties of Fibonacci numbers and it's generalizations are
known, these relations are very important. Using these relations, properties
of Lucas numbers and properties of it's generalizations can be obtained. In
addition to obtaining these relations, we give a generalized Binet formula
and combinatorial representation for $k$SO$k$L with the help of properties
of generalized Fibonacci numbers.\bigskip 

\subsection{Fibonacci and Lucas Numbers and Properties of Fibonacci
Generalization}

\label{sect1} \bigskip \qquad

The well-known Fibonacci sequence $\left\{ f_{n}\right\} $ is defined
recursively by the equation, 
\begin{equation*}
f_{n}=f_{n-1}+f_{n-2},\ \text{for }n\geq 3\ \ 
\end{equation*}%
where $f_{1}=1,$ $f_{2}=1$ and \bigskip Lucas sequence $\left\{
l_{n}\right\} $ is defined recursively by the equation,%
\begin{equation*}
l_{n}=l_{n-1}+l_{n-2},\text{\ for }n\geq 2
\end{equation*}%
where $l_{0}=2,$ $l_{1}=1.$

\bigskip Miles [10] defined generalized order-$k$ Fibonacci numbers(GO$k$F)
as,%
\begin{equation}
f_{k,n}=\sum\limits_{j=1}^{k}f_{k,n-j}\ 
\end{equation}%
for $n>k\geq 2$, with boundary conditions: $f_{k,1}=f_{k,2}=f_{k,3}=\cdots
=f_{k,k-2}=0,$ $f_{k,k-1}=f_{k,k}=1.$\newline
Er [1] defined $k$SO$k$F as; for $n>0,$ $1\leq i\leq k$%
\begin{equation}
f_{k,n}^{\text{ }i}=\sum\limits_{j=1}^{k}c_{j}f_{k,n-j}^{\text{ }i}\ \ 
\end{equation}%
with boundary conditions for $1-k\leq n\leq 0,$

\begin{equation*}
f_{k,n}^{\text{ }i}=\left\{ 
\begin{array}{c}
1\text{ \ \ \ \ \ if \ }i=1-n, \\ 
0\text{ \ \ \ \ \ \ \ \ \ otherwise,}%
\end{array}%
\right.
\end{equation*}%
where $c_{j}$ $(1\leq j\leq k)$ are constant coefficients, $f_{k,n}^{\text{ }%
i}$ is the $n$-th term of $i$-th sequence of order $k$ generalization. $k$%
-th column of this generalization involves the Miles generalization for $%
i=k, $ i.e. $f_{k,n}^{k}=f_{k,k+n-2}.$

\bigskip Er [1] showed 
\begin{equation*}
F_{n+1}^{\sim }=AF_{n}^{\sim }
\end{equation*}%
where%
\begin{equation*}
A=\left[ 
\begin{array}{cccccc}
c_{1} & c_{2} & c_{3} & \cdots & c_{k-1} & c_{k} \\ 
1 & 0 & 0 & \cdots & 0 & 0 \\ 
0 & 1 & 0 & \cdots & 0 & 0 \\ 
\vdots & \vdots & \vdots & \ddots & \vdots & \vdots \\ 
0 & 0 & 0 & \cdots & 1 & 0%
\end{array}%
\right]
\end{equation*}%
is $k\times k$\ \ companion matrix and%
\begin{equation}
\ F_{n}^{\sim }=\left[ 
\begin{array}{cccc}
f_{k,n}^{1} & f_{k,n}^{2} & \cdots & f_{k,n}^{k} \\ 
f_{k,n-1}^{1} & f_{k,n-1}^{2} & \cdots & f_{k,n-1}^{k} \\ 
\vdots & \vdots & \ddots & \vdots \\ 
f_{k,n-k+1}^{1} & f_{k,n-k+1}^{2} & \cdots & f_{k,n-k+1}^{k}%
\end{array}%
\right]
\end{equation}%
is $k\times k$\ matrix.

Karaduman [5] showed $F_{1}^{\sim }=A$ and $F_{n}^{\sim }=A^{n}$ for $%
c_{j}=1,$ $(1\leq j\leq k)$ .

Kalman [2] derived the Binet formula by using Vandermonde matrix, for $%
\lambda _{i}$ $(1\leq i\leq k)$ are roots of the polynomial 
\begin{equation}
P(x;t_{1},t_{2},\ldots ,t_{k})=x^{k}-t_{1}x^{k-1}-\cdots -t_{k}
\end{equation}%
($t_{1,\ldots ,}t_{k}$ are constants) 
\begin{equation}
f_{k,n}^{k}=\sum\limits_{i=1}^{k}\frac{(\lambda _{i})^{n}}{P^{^{\prime
}}(\lambda _{i})}
\end{equation}%
where $f_{k,n}^{k}$ is $($for $c_{j}=1,$ $1\leq j\leq k$ and $i=k)$ $k$-th
sequences of $k$SO$k$F and $P%
%TCIMACRO{\U{b4}}%
%BeginExpansion
{\acute{}}%
%EndExpansion
(x)$ is derivative of the polynomial (1.4).

\bigskip

K\i l\i \c{c} [5] studied $F_{n}^{\sim }$ and $f_{k,n}^{k}$ and gave some
formulas and properties concerning $k$SO$k$F. One of these is Binet formula
for $k$SO$k$F. For roots of $(1.4)$ named as $\lambda _{i}$ $(1\leq i\leq k)$%
,

\begin{equation}
V=\left[ 
\begin{array}{cccc}
\lambda _{1}^{k-1} & \lambda _{2}^{k-1} & \ldots & \lambda _{k}^{k-1} \\ 
\lambda _{1}^{k-2} & \lambda _{2}^{k-2} & \ldots & \lambda _{k}^{k-2} \\ 
\vdots & \vdots & \ddots & \vdots \\ 
\lambda _{1} & \lambda _{2} & \ldots & \lambda _{k} \\ 
1 & 1 & \ldots & 1%
\end{array}%
\right] \text{ and }d_{k}^{i}=\left[ 
\begin{array}{c}
\lambda _{1}^{k-i+n} \\ 
\lambda _{2}^{k-i+n} \\ 
\vdots \\ 
\lambda _{k}^{k-i+n}%
\end{array}%
\right]
\end{equation}
where $V$ \ is a $k\times k$ Vandermonde matrix and $V_{j}^{(i)}$ is a $%
k\times k$ matrix obtained from $V$ by replacing $j$-th column of $V$ by $%
d_{k}^{i} $, Binet formula of $f_{k,n}^{k}$ is;

\begin{equation}
f_{k,n}^{k}=t_{1k}=\frac{\text{det}(V_{k}^{(1)})}{\text{det(}V\text{)}}.
\end{equation}

\subsection{Generalized Fibonacci and Lucas Polynomials}

MacHenry [7] defined generalized Fibonacci polynomials $(F_{k,n}(t))$, Lucas
polynomials $(G_{k,n}(t))$ and obtained important relations between
generalized Fibonacci and Lucas polynomials, where $t_{i}$ $(1\leq i\leq k)$
are constant coefficients of the core polynomial $(1.4).$ $F_{k,n}(t)$
defined inductively by 
\begin{eqnarray*}
F_{k,n}(t) &=&0,\text{ }n<0 \\
F_{k,0}(t) &=&1 \\
F_{k,1}(t) &=&t_{1} \\
F_{k,n+1}(t) &=&t_{1}F_{k,n}(t)+\cdots +t_{k}F_{k,n-k+1}(t)
\end{eqnarray*}%
where $t=(t_{1},t_{2},\ldots ,t_{k}),$ $k\in 
%TCIMACRO{\U{2115} }%
%BeginExpansion
\mathbb{N}
%EndExpansion
,$ $n$ is an integer and $G_{k,n}(t_{1},t_{2},\ldots ,t_{k})$ defined by%
\begin{eqnarray}
G_{k,n}(t) &=&0,\text{ }n<0 \\
G_{k,0}(t) &=&\text{ }k  \notag \\
G_{k,1}(t) &=&t_{1}\text{ }  \notag \\
G_{k,n+1}(t) &=&t_{1}G_{k,n}(t)+\cdots +t_{k}G_{k,n-k+1}(t).  \notag
\end{eqnarray}

In addition, in [9] authors obtained $F_{k,n}(t)$ and $G_{k,n}(t)$ ($n,k\in 
%TCIMACRO{\U{2115} }%
%BeginExpansion
\mathbb{N}
%EndExpansion
$, $n\geq 1)$ as

\begin{equation}
F_{k,n}(t)=\sum\limits_{a\vdash n}\binom{\left\vert a\right\vert }{%
a_{1,\ldots ,}a_{k}}t_{1}^{a_{1}}\ldots t_{k}^{a_{k}}
\end{equation}%
and

\begin{equation}
G_{k,n}(t)=\sum\limits_{a\vdash n}\frac{n}{\left\vert a\right\vert }\binom{%
\left\vert a\right\vert }{a_{1,\ldots ,}a_{k}}t_{1}^{a_{1}}\ldots
t_{k}^{a_{k}}
\end{equation}%
where $a_{i}$ are nonnegative integers for all $i$ $(1\leq i\leq k),$ with
initial conditions given by%
\begin{equation*}
F_{k,0}(t)=1,\text{ }F_{k,-1}(t)=0,\text{ }\ldots ,\text{ }F_{k,-k+1}(t)=0
\end{equation*}%
and%
\begin{equation*}
G_{k,0}(t)=k,\text{ }G_{k,-1}(t)=0,\text{ }\cdots ,\text{ }G_{k,-k+1}(t)=0.
\end{equation*}

In this paper, the notations $a\vdash n$ and $\left\vert a\right\vert $ are
used instead of $\sum\limits_{j=1}^{k}ja_{j}=n$ and $\sum%
\limits_{j=1}^{k}a_{j},$ respectively. A combinatorial representation for
Fibonacci polynomials is given in [9] as 
\begin{equation}
F_{2,n}(t)=\sum\limits_{j=0}^{\left\lceil \frac{n}{2}\right\rceil }(-1)^{j}%
\binom{n-j}{j}F_{1}^{n-2j}(-t_{2})^{j}
\end{equation}%
for $n\in 
%TCIMACRO{\U{2124} }%
%BeginExpansion
\mathbb{Z}
%EndExpansion
,$ where $\left\lceil \frac{n}{2}\right\rceil =k,$ either $n=2k$ or $n=2k-1.$

In [8], matrices $A_{(k)}^{\infty }$ and $D_{(k)}^{\infty }$ are defined by
using the following matrix,

\begin{equation*}
A_{(k)}=\left[ 
\begin{array}{ccccc}
0 & 1 & 0 & \ldots & 0 \\ 
0 & 0 & 1 & \ldots & 0 \\ 
\vdots & \vdots & \vdots & \ddots & \vdots \\ 
0 & 0 & 0 & \ldots & 1 \\ 
t_{k} & t_{k-1} & t_{k-2} & \ldots & t_{1}%
\end{array}%
\right] \text{ .}
\end{equation*}

They also record the orbit of the $k$-th row vector of $A_{(k)}$ under the
action of $A_{(k)}$, below $A_{(k)}$, and the orbit of the first row of $%
A_{(k)}$ under the action of $A_{(k)}^{-1}$ on the first row of \ $A_{(k)}$
is recorded above $A_{(k)},$ and consider the $\infty \times k$ matrix whose
row vectors are the elements of the doubly infinite orbit of $A_{(k)}$
acting on any one of them. For $k=3,$ $A_{(k)}^{\infty }$ looks like this%
\begin{equation*}
A_{(3)}^{\infty }=\left[ 
\begin{array}{ccc}
\cdots & \cdots & \cdots \\ 
S_{(-n,1^{2})} & -S_{(-n,1)} & S_{(-n)} \\ 
\cdots & \cdots & \cdots \\ 
S_{(-3,1^{2})} & -S_{(-3,1)} & S_{(-3)} \\ 
1 & 0 & 0 \\ 
0 & 1 & 0 \\ 
0 & 0 & 1 \\ 
t_{3} & t_{2} & t_{1} \\ 
\cdots & \cdots & \cdots \\ 
S_{(n-1,1^{2})} & -S_{(n-1,1)} & S_{(n-1)} \\ 
S_{(n,1^{2})} & -S_{(n,1)} & S_{(n)} \\ 
\cdots & \cdots & \cdots%
\end{array}%
\right] \bigskip \ \ \ \ 
\end{equation*}%
and

\ \ 
\begin{equation*}
A_{(k)}^{n}=\left[ 
\begin{array}{ccccc}
(-1)^{k-1}S_{(n-k+1,1^{k-1})} & \cdots & (-1)^{k-j}S_{(n-k+1,1^{k-j})} & 
\cdots & S_{(n-k+1)} \\ 
\cdots & \cdots & \cdots & \cdots & \cdots \\ 
(-1)^{k-1}S_{(n,1^{k-1})} & \cdots & (-1)^{k-j}S_{(n,1^{k-j})} & \cdots & 
S_{(n)}%
\end{array}%
\right] \ \ \ 
\end{equation*}%
where%
\begin{equation}
\ \ S_{(n-r,1^{r})}=(-1)^{r}\sum\limits_{j=r+1}^{n}t_{j}S_{(n-j)},\ 0\leq
r\leq n.
\end{equation}

Derivative of the core polynomial $P(x;t_{1},t_{2},\ldots
,t_{k})=x^{k}-t_{1}x^{k-1}-\cdots -t_{k}$ is $P%
%TCIMACRO{\U{b4}}%
%BeginExpansion
{\acute{}}%
%EndExpansion
(x)=kx^{k-1}-t_{1}(k-1)x^{k-2}-\cdots -t_{k-1},$ which is represented by the
vector ($-t_{k-1},\ldots ,-t_{1}(k-1),k)$ and the orbit of this vector under
the action of $A_{(k)}$ gives the standard matrix representation $%
D_{(k)}^{\infty }.$

Right hand column of $A_{(k)}^{\infty }$ contains sequence of the
generalized Fibonacci polynomials $F_{k,n}(t)$ and $%
tr(A_{(k)}^{n})=G_{k,n}(t)$ for $n\in 
%TCIMACRO{\U{2124} }%
%BeginExpansion
\mathbb{Z}
%EndExpansion
,$ where $G_{k,n}(t)$ is the sequence of the generalized Lucas polynomials,
which is also a $t$-linear recursion. In addition, the right hand column of $%
D_{(k)}^{\infty }$ contains sequence of the generalized Lucas polynomials $%
G_{k,n}(t)$.

It is clear that, for $t_{i}=1$ and $c_{i}=1$ $(1\leq i\leq k)$ $%
S_{(n)}=f_{k,n}^{1}$ where $f_{k,n}^{1},$ is the $n$-th term of the first\
sequence of $k$SO$k$F. Moreover, the matrix $A_{(k)}^{\infty }$ involves the
generalization (1.2).

\begin{exam}
We give matrix $A_{(k)}^{\infty }$ for k=3 and the matrix $D_{(k)}^{\infty }$
for $k=4$, while $t_{1}=t_{2}=\cdots =t_{k}=1$%
\begin{equation*}
A_{(3)}^{\infty }=\left[ 
\begin{array}{ccc}
\ldots & \ldots & \ldots \\ 
1 & 0 & 0 \\ 
0 & 1 & 0 \\ 
0 & 0 & 1 \\ 
1 & 1 & 1 \\ 
\ldots & \ldots & \ldots%
\end{array}%
\ \ \right] \ \text{and }D_{(4)}^{\infty }=\left[ 
\begin{array}{cccc}
\ldots & \ldots & \ldots & \ldots \\ 
7 & 1 & 0 & -1 \\ 
-1 & 6 & 0 & -1 \\ 
-1 & -2 & 5 & -1 \\ 
-1 & -2 & -3 & 4 \\ 
\ldots & \ldots & \ldots & \ldots%
\end{array}%
\right]
\end{equation*}
\end{exam}

\section{Generalizations of Lucas Numbers}

\label{ns} \bigskip \bigskip For $t_{s}=1,$ $1\leq s\leq k,$ the Lucas
polynomials $G_{k,n}(t)$ and $D_{(k)}^{\infty }$ together are reduced to%
\begin{equation}
l_{k,n}=\sum\limits_{j=1}^{k}l_{k,n-j}
\end{equation}%
with boundary conditions%
\begin{equation*}
l_{k,1-k}=l_{k,2-k}=\ldots =l_{k,-1}=-1\ \text{and }l_{k,0}=k,
\end{equation*}
which is called generalized order-$k$ Lucas numbers(GO$k$L). When k=2, it is
reduced to ordinary Lucas numbers.

In this paper, we study on positive direction of $D_{(k)}^{\infty }$ for $%
t_{s}=1$, $1\leq s\leq k,$ which can be written explicitly as

\begin{equation}
\ l_{k,n}^{i}=\sum\limits_{j=1}^{k}l_{k,n-j}^{i}
\end{equation}%
for $n>0$ and $1\leq i\leq k,$\ with boundary conditions%
\begin{equation*}
l_{k,n}^{i}=\left\{ 
\begin{array}{lc}
-i\ \ \ \ \ \ \ \ \text{ \ if }i-n<k, &  \\ 
-2n+i\ \text{ \ if }i-n=k, &  \\ 
k-i-1\text{ \ if }i-n>k & 
\end{array}%
\right. \ \ \ 
\end{equation*}%
for $1-k\leq n\leq 0,$ where $l_{k,n}^{i}$ is the $n$-th term of $i$-th
sequence. This generalization is called $k$ sequences of the generalized
order-$k$ Lucas numbers($k$SO$k$L).

Although names are the same, the initial conditions of this generalization
are different from the generalizations in [3]. These initial conditions
arise from Lucas Polynomials and $D_{(k)}^{\infty }.$

When $i=k=2,$ we obtain ordinary Lucas numbers and $l_{k,n}^{k}=l_{k,n}$.

\begin{exam}
Substituting $k=3$ and $i=2$ we obtain the generalized order-$3$ Lucas
sequence as;%
\begin{equation*}
l_{3,-2}^{2}=0,\text{ }l_{3,-1}^{2}=4,\text{ }l_{3,0}^{2}=-2,\text{ }%
l_{3,1}^{2}=2,\text{ }l_{3,2}^{2}=4,\text{ }l_{3,3}^{2}=4,\text{ }\ldots
\end{equation*}
\end{exam}

\begin{lem}
Matrix multiplication and (2.2) can be used to obtain 
\begin{equation*}
L_{n+1}^{\sim }=A_{1}L_{n}^{\sim }
\end{equation*}%
where\ 
\begin{equation}
A_{1}=\left[ 
\begin{array}{cccccc}
1 & 1 & 1 & \ldots & 1 & 1 \\ 
1 & 0 & 0 & \ldots & 0 & 0 \\ 
0 & 1 & 0 & \cdots & 0 & 0 \\ 
\vdots & \vdots & \vdots & \ddots & \vdots & \vdots \\ 
0 & 0 & 0 & \ldots & 1 & 0%
\end{array}%
\right] _{k\times k}=\left[ 
\begin{array}{cccc}
1 & 1 & \ldots & 1 \\ 
&  &  & 0 \\ 
& I &  & \vdots \\ 
&  &  & 0%
\end{array}%
\right] _{k\times k}\ \ \ \ \ \ \ 
\end{equation}%
where $I$ is ($k-1)\times (k-1)$ identity matrix and\ \ we define a $k\times
k$ matrix $L_{n}^{\sim }$ as;%
\begin{equation}
L_{n}^{\sim }=\left[ 
\begin{array}{cccc}
l_{k,n}^{1} & l_{k,n}^{2} & \ldots & l_{k,n}^{k} \\ 
l_{k,n-1}^{1} & l_{k,n-1}^{2} & \ldots & l_{k,n-1}^{k} \\ 
\vdots & \vdots & \ddots & \vdots \\ 
l_{k,n-k+1}^{1} & l_{k,n-k+1}^{2} & \ldots & l_{k,n-k+1}^{k}%
\end{array}%
\right] \ 
\end{equation}%
which is contained by $k\times k$ block of\ $D_{(k)}^{\infty }$ for $t_{i}=1$%
, $1\leq i\leq k.$\ \ \ \ \ 
\end{lem}

\begin{lem}
Let $A_{1}$ and $L_{n}^{\sim }$ be as in (2.3) and (2.4), respectively. Then
\linebreak $L_{n+1}^{\sim }=A_{1}^{n+1}L_{0}^{\sim },$ where\ \ \ \ \ \ 
\begin{equation*}
\ L_{0}^{\sim }=\left[ 
\begin{array}{ccccccc}
-1 & -2 & -3 & \ldots & -(k-2) & -(k-1) & k \\ 
-1 & -2 & -3 & \ldots & -(k-2) & k+1 & -1 \\ 
\vdots & \vdots & \vdots & \ldots & k+2 & 0 & -1 \\ 
-1 & -2 & 2k-3 & \ldots & 1 & 0 & -1 \\ 
-1 & 2k-2 & k-4 & \ldots & \vdots & \vdots & \vdots \\ 
2k-1 & k-3 & k-4 & \ldots & 1 & 0 & -1%
\end{array}%
\right] _{k\times k}.
\end{equation*}%
\ \ \ \ \ \ \ \ \ \ \ \ \ \ \ \ \ \ \ \ \ \ \ \ \ \ \ \ \ \ \ \ \ \ \ \ \ \
\ \ \ \ \ \ \ \ \ \ \ \ \ \ \ \ \ \ \ \ \ \ \ \ \ \ \ \ \ \ \ \ \ \ \ \ \ \
\ \ \ \ \ \ \ \ \ \ \ \ \ \ \ \ \ \ \ \ \ \ 
\end{lem}

\begin{proof}
It is clear that $L_{1}^{\sim }=A_{1}L_{0}^{\sim }$ and $L_{n+1}^{\sim
}=A_{1}L_{n}^{\sim }.$ So by induction and properties of matrix
multiplication, we have $L_{n+1}^{\sim }=A^{n+1}L_{0}^{\sim }.$\bigskip
\end{proof}

\begin{lem}
Let $F_{n}^{\sim }$ and $L_{n}^{\sim }$ be as in (1.3) and (2.4),
respectively. Then%
\begin{equation*}
L_{n}^{\sim }=F_{n}^{\sim }L_{0}^{\sim }.
\end{equation*}
\end{lem}

\begin{proof}
Proof is trivial from $F_{n}^{\sim }=A_{1}^{n}($see $\left[ 4\right] )$ and
Lemma 2.3.
\end{proof}

\begin{exam}
From Lemma 2.4 for $k=2$, we have%
\begin{equation*}
\left[ 
\begin{array}{cc}
l_{2,n}^{1} & l_{2,n}^{2} \\ 
l_{2,n-1}^{1} & l_{2,n-1}^{2}%
\end{array}%
\right] =\left[ 
\begin{array}{cc}
f_{2,n}^{1} & f_{2,n}^{2} \\ 
f_{2,n-1}^{1} & f_{2,n-1}^{2}%
\end{array}%
\right] \left[ 
\begin{array}{cc}
-1 & 2 \\ 
3 & -1%
\end{array}%
\right] .\ 
\end{equation*}%
Therefore, $l_{2,n}^{2}=2f_{2,n}^{1}-f_{2,n}^{2}$. Since $%
f_{2,n}^{1}=f_{2,n+1}^{2}$ for all $n\in 
%TCIMACRO{\U{2124} }%
%BeginExpansion
\mathbb{Z}
%EndExpansion
,$ then we have%
\begin{equation*}
l_{2,n}^{2}=2f_{2,n+1}^{2}-f_{2,n}^{2}
\end{equation*}%
where $l_{2,n}^{2}$ and $f_{2,n}^{2}$ are ordinary Lucas and Fibonacci
numbers, respectively.\linebreak\ \ \ For $k=3$, we have%
\begin{equation*}
\left[ 
\begin{array}{ccc}
l_{3,n}^{1} & l_{3,n}^{2} & l_{3,n}^{3} \\ 
l_{3,n-1}^{1} & l_{3,n-1}^{2} & l_{3,n-1}^{3} \\ 
l_{3,n-2}^{1} & l_{3,n-2}^{2} & l_{3,n-2}^{3}%
\end{array}%
\right] =\left[ 
\begin{array}{ccc}
f_{3,n}^{1} & f_{3,n}^{2} & f_{3,n}^{3} \\ 
f_{3,n-1}^{1} & f_{3,n-1}^{2} & f_{3,n-1}^{3} \\ 
f_{3,n-2}^{1} & f_{3,n-2}^{2} & f_{3,n-2}^{3}%
\end{array}%
\right] \left[ 
\begin{array}{ccc}
-1 & -2 & 3 \\ 
-1 & 4 & -1 \\ 
5 & 0 & -1%
\end{array}%
\right] .
\end{equation*}%
Therefore, $l_{3,n}^{3}=3f_{3,n}^{1}-f_{3,n}^{2}-f_{3,n}^{3}.$ Since for
k=3, $f_{3,n}^{1}=f_{3,n+1}^{3}$ and \newline
$f_{3,n}^{2}=f_{3,n-1}^{1}+f_{3,n-1}^{3}=f_{3,n}^{3}+f_{3,n-1}^{3}$ for all $%
n\in 
%TCIMACRO{\U{2124} }%
%BeginExpansion
\mathbb{Z}
%EndExpansion
,$ we have%
\begin{equation*}
l_{3,n}^{3}=3f_{3,n+1}^{3}-2f_{3,n}^{3}-f_{3,n-1}^{3}.
\end{equation*}
\end{exam}

\begin{thrm}
For $i=k$, $n\geq 0$ and $c_{1}=\cdots =c_{k}=1,$%
\begin{equation}
\ l_{k,n}^{k}=kf_{k,n+1}^{k}-(k-1)f_{k,n}^{k}-\cdots
-f_{k,n-k+2}^{k}=kf_{k,n+1}^{k}-\sum\limits_{j=2}^{k}(k-j+1)f_{k,n+2-j}^{k}
\end{equation}%
where $l_{k,n}^{i}$ and $f_{k,n}^{\text{ }i}$ $k$SO$k$L and $k$SO$k$F,
respectively.\bigskip
\end{thrm}

\begin{proof}
\bigskip We use mathematical induction to prove the following equality%
\begin{equation*}
l_{k,n}^{k}=kf_{k,n+1}^{k}-\sum\limits_{j=2}^{k}(k-j+1)f_{k,n+2-j}^{k}.
\end{equation*}%
It is easy to obtain $l_{k,0}^{k}=k,$ $f_{k,0}^{k}=0$ and $f_{k,1}^{k}=1$
for all $k\in 
%TCIMACRO{\U{2124} }%
%BeginExpansion
\mathbb{Z}
%EndExpansion
^{+}$ with $k\geq 2,$ from the definition of $k$SO$k$L and $k$SO$k$F. So,
the equation (2.5) is true for $n=0$, i.e.,%
\begin{equation*}
l_{k,0}^{k}=kf_{k,1}^{k}-(k-1)f_{k,0}^{k}-\cdots -f_{k,-k+2}^{k}=k.1+0=k.
\end{equation*}%
Suppose that the equation holds for all positive integers less than or equal
to $n$ i.e., for integer $n$%
\begin{equation*}
l_{k,n}^{k}=kf_{k,n+1}^{k}-\sum\limits_{j=2}^{k}(k-j+1)f_{k,n+2-j}^{k}
\end{equation*}%
then from (1.2) and (2.2), for $c_{1}=\cdots =c_{k}=1,$ we get;\bigskip 
\begin{eqnarray*}
l_{k,n+1}^{k} &=&l_{k,n}^{k}+l_{k,n-1}^{k}+l_{k,n-2}^{k}+\cdots
+l_{k,n-k+1}^{k} \\
&=&(kf_{k,n+1}^{k}-(k-1)f_{k,n}^{k}-\cdots -f_{k,n-k+2}^{k})+ \\
&&(kf_{k,n}^{k}-(k-1)f_{k,n-1}^{k}-\cdots -f_{k,n-k+1}^{k})+ \\
&&\ldots +(kf_{k,n-k+2}^{k}-(k-1)f_{k,n-k+1}^{k}-\cdots -f_{k,n-2k+3}^{k}) \\
&=&kf_{k,n+2}^{k}-(k-1)f_{k,n+1}^{k}-\cdots -f_{k,n-k+3}^{k} \\
&=&kf_{k,n+2}^{k}-\sum\limits_{j=2}^{k}(k-j+1)f_{k,n+3-j}^{k}\text{.}
\end{eqnarray*}%
\bigskip \bigskip So, the equation holds for $(n+1)$ and proof is complete.
\end{proof}

Since $f_{k,n}^{k}=f_{k,n+k-2}$ and $l_{k,n}^{k}=l_{k,n}$ the following
relation is obvious%
\begin{equation*}
l_{k,n}=kf_{k,n+k-1}-\sum\limits_{j=2}^{k}(k-j+1)f_{k,n+k-j}
\end{equation*}%
where $f_{k,n}$ is the $n$-th GO$k$F as in (1.1), $\ l_{k,n}$ is GO$k$L as
in (2.1) and $l_{k,n}^{k}$ is the $n$-th term of $k$-th sequences of the $k$%
SO$k$L as in (2.2).

The following theorem shows that equation (2.5) is valid for Generalized
Fibonacci and Lucas Polynomials as well.

\begin{thrm}
\bigskip For $k\geq 2$ and $n\geq 0,$%
\begin{equation*}
G_{k,n}(t)=kF_{k,n}(t)-\sum\limits_{j=2}^{k}(k-j+1)t_{j-1}F_{k,n+1-j}(t)
\end{equation*}%
where $F_{k,n}(t)$ and $G_{k,n}(t)$ are the Generalized Fibonacci and Lucas
Polynomials, respectively.
\end{thrm}

\begin{proof}
Proof is by induction as Theorem 2.6.
\end{proof}

\begin{thrm}
For $i=k$ and $n\geq 0,$%
\begin{equation}
l_{k,n}^{k}=\sum\limits_{j=1}^{k}jf_{k,n+1-j}^{k}
\end{equation}%
where $l_{k,n}^{i}$ and $f_{k,n}^{\text{ }i}$ are the $k$SO$k$L and $k$SO$k$%
F respectively.
\end{thrm}

\begin{proof}
Proof is by induction as Theorem 2.6.
\end{proof}

\begin{lem}
For\ $k\geq 2,i$-th sequences of $k$SO$k$L in terms of $k$-th sequences of $%
k $SO$k$L is 
\begin{equation}
l_{k,n}^{i}=\left\{ 
\begin{array}{l}
l_{k,n-1}^{k}\text{ }\ \ \ \ \ \ \ \ \ \ \ \text{if }i=1 \\ 
\sum\limits_{m=1}^{i}l_{k,n-m}^{k}\ \ \ \text{ \ if }1<i<k \\ 
l_{k,n}^{k}\text{ }\ \ \ \ \ \ \ \ \ \ \ \ \ \ \text{if }i=k%
\end{array}%
\right. .
\end{equation}
\end{lem}

\bigskip

\begin{thrm}
$i$-th sequences of $k$SO$k$L can be written in terms of $k$-th sequences of 
$k$SO$k$F (which is GO$k$F with index iteration) in different ways;
\end{thrm}

$i)$For $k\geq 3$ and $1\leq i\leq k$ 
\begin{equation*}
l_{k,n}^{i}=\sum\limits_{j=1}^{k+i-1}d_{j}f_{k,n-j}^{k}\text{ }
\end{equation*}%
where $1\leq i\leq k$, $n\geq 0$ and constant coefficient 
\begin{equation*}
d_{j}=\left\{ 
\begin{array}{l}
\frac{j(j+1)}{2}\text{ }\ \ \ \ \ \ \ \ \ \ \ \ \ \ \ \ \ \ \ \ \ \ \ \ \ \ 
\text{if }1\leq j\leq i \\ 
\frac{j(j+1)}{2}-\frac{(j-i)(j-i+1)}{2}\text{ \ \ \ \ \ \ \ \ if }i+1\leq
j\leq k-1 \\ 
\frac{k(k+1)}{2}-\frac{(j-i)(j-i+1)}{2}\ \ \ \ \ \ \ \ \text{if }k\leq j\leq
k+i-1%
\end{array}%
\right.
\end{equation*}

$ii)$For $k\geq 2$ and $1\leq i\leq k$%
\begin{equation*}
l_{k,n}^{i}=\left\{ 
\begin{array}{l}
kf_{k,n}^{k}-\sum\limits_{j=2}^{k}(k-j+1)f_{k,n+1-j}^{k}\text{ \ \ \ \ \ \ \
\ \ \ \ \ \ \ \ \ \ \ \ \ \ \ \ \ \ if }i=1 \\ 
\sum\limits_{m=1}^{i}kf_{k,n-m+1}^{k}-\sum\limits_{m=1}^{i}\sum%
\limits_{j=2}^{k}(k-j+1)f_{k,n-m-j+2}^{k}\text{ \ \ \ if }1<i<k \\ 
kf_{k,n+1}^{k}-\sum\limits_{j=2}^{k}(k-j+1)f_{k,n+2-j}^{k}\text{ }\ \ \ \ \
\ \ \ \ \ \ \ \ \ \ \ \ \ \ \ \ \ \text{if }i=k%
\end{array}%
\right.
\end{equation*}

$iii)$For $k\geq 2$ and $1\leq i\leq k$%
\begin{equation*}
l_{k,n}^{i}=\left\{ 
\begin{array}{l}
\sum\limits_{j=1}^{k}jf_{k,n-j}^{k}\text{ \ \ \ \ \ \ \ \ \ \ \ \ \ \ \ if }%
i=1 \\ 
\sum\limits_{m=1}^{i}\sum\limits_{j=1}^{k}jf_{k,n-m-j+1}^{k}\text{ \ \ if }%
1<i<k \\ 
\sum\limits_{j=1}^{k}jf_{k,n+1-j}^{k}\text{ \ \ \ \ \ \ \ \ \ \ \ \ if }i=k%
\end{array}%
\right..
\end{equation*}

\begin{proof}
$i)$Proof is from Theorem 2.8 and Lemma 2.9.

$ii)$Proof is from Theorem 2.6 and Lemma 2.9.

$iii)$Proof is from Theorem 2.8 and Lemma 2.9.
\end{proof}

\begin{exam}
Let us obtain $l_{k,n}^{i}$ for $k=4$, $n=4$ and $i=3$ by using Theorem
(2.10 $iii$).%
\begin{eqnarray*}
l_{4,4}^{3}
&=&\sum\limits_{m=1}^{3}\sum\limits_{j=1}^{4}j.f_{4,4-m-j+1}^{4}=\sum%
\limits_{m=1}^{3}(f_{4,4-m}^{4}+2f_{4,3-m}^{4}+3f_{4,2-m}^{4}+4f_{4,1-m}^{4})
\\
&=&f_{4,3}^{4}+2f_{4,2}^{4}+3f_{4,1}^{4}+4f_{4,0}^{4}+f_{4,2}^{4}+2f_{4,1}^{4}+f_{4,1}^{4}=11
\end{eqnarray*}%
since $f_{4,0}^{4}=0,$ $f_{4,1}^{4}=f_{4,2}^{4}=1$ and $f_{4,3}^{4}=2$.
\end{exam}

\begin{thrm}
Let $l_{k,n}^{i}$ and $f_{k,n}^{i}$ be the $k$SO$k$L and $k$SO$k$F,
respectively. Then, for $m,n\in 
%TCIMACRO{\U{2124} }%
%BeginExpansion
\mathbb{Z}
%EndExpansion
$ and $1\leq i\leq k-1$,%
\begin{equation*}
l_{n+m}^{i}=\sum_{j=1}^{i}(l_{m-j}^{k}\sum_{s=1}^{j}f_{n}^{s})+%
\sum_{j=i+1}^{k}(l_{m-j}^{k}\sum_{s=j-i+1}^{j}f_{n}^{s})+%
\sum_{j=k+1}^{k+i-1}(l_{m-j}^{k}\sum_{s=j-i+1}^{k}f_{n}^{s})
\end{equation*}%
where we assume that, the sum is equal to zero, if the subscript is greater
than the superscript in the sum.
\end{thrm}

\begin{proof}
\bigskip We know that $L_{n}^{\sim }=F_{n}^{\sim }L_{0}^{\sim }$ (Lemma
2.4), so we can write that%
\begin{equation*}
L_{n+m}^{\sim }=F_{n+m}^{\sim }L_{0}^{\sim }=A_{1}^{n+m}L_{0}^{\sim
}=A_{1}^{n}A_{1}^{m}L_{0}^{\sim }=A_{1}^{n}L_{m}^{\sim }=F_{n}^{\sim
}L_{m}^{\sim }.
\end{equation*}
From this matrix product and Lemma 2.9 we obtain 
\begin{eqnarray*}
l_{k,n+m}^{i} &=&f_{k,n}^{1}l_{k,m}^{i}+\cdots +f_{k,n}^{k}l_{k,m-k+1}^{i} \\
&=&f_{k,n}^{1}(l_{k,m-1}^{k}+\cdots +l_{k,m-i}^{k})+\cdots
+f_{k,n}^{k}(l_{k,m-k}^{k}+\cdots +l_{k,m-k-i-1}^{k}) \\
&=&l_{k,m-1}^{k}f_{n}^{1}+l_{k,m-2}^{k}(f_{k,n}^{1}+f_{k,n}^{2})+\cdots
l_{k,m-i}^{k}(f_{k,n}^{1}+f_{k,n}^{2}+\cdots +f_{k,n}^{\text{ }i})+ \\
&&l_{k,m-i-1}^{k}(f_{k,n}^{2}+f_{k,n}^{3}+\cdots +f_{k,n}^{\text{ }%
i+1})+\cdots +l_{k,m-k}^{k}(f_{k,n}^{k-i+1}+\cdots +f_{k,n}^{k})+ \\
&&l_{k,m-k-1}^{k}(f_{k,n}^{k-i+2}+\cdots +f_{k,n}^{k})+\cdots
+l_{k,m-k-i-1}^{k}f_{k,n}^{k} \\
&=&\sum_{j=1}^{i}(l_{k,m-j}^{k}\sum_{t=1}^{j}f_{k,n}^{t})+%
\sum_{j=i+1}^{k}(l_{k,m-j}^{k}\sum_{t=j-i+1}^{j}f_{k,n}^{t})+%
\sum_{j=k+1}^{k+i-1}(l_{k,m-j}^{k}\sum_{t=j-i+1}^{k}f_{k,n}^{t}).
\end{eqnarray*}

\begin{exam}
Let us obtain $l_{k,n+m}^{i}$ for $k=5$, $i=3$, $n=3$ and $m=4$ by using
Theorem 2.12;%
\begin{eqnarray*}
l_{5,3+4}^{3}
&=&l_{7}^{3}=\sum_{j=1}^{3}(l_{5,4-j}^{5}\sum_{t=1}^{j}f_{5,3}^{t})+%
\sum_{j=4}^{5}(l_{5,4-j}^{5}\sum_{t=j-2}^{j}f_{5,3}^{t})+%
\sum_{j=6}^{7}(l_{5,4-j}^{5}\sum_{t=j-2}^{5}f_{5,3}^{t}) \\
&=&l_{5,3}^{5}f_{5,3}^{1}+l_{5,2}^{5}(f_{5,3}^{1}+f_{5,3}^{2})+l_{5,1}^{5}(f_{5,3}^{1}+f_{5,3}^{2}+f_{5,3}^{3})+l_{5,0}^{5}(f_{5,3}^{2}+f_{5,3}^{3}+f_{5,3}^{4})
\\
&&+l_{5,-1}^{5}(f_{5,3}^{3}+f_{5,3}^{4}+f_{5,3}^{5})+l_{5,-2}^{5}(f_{5,3}^{4}+f_{5,3}^{5})+l_{5,-3}^{5}f_{5,3}^{5}
\\
&=&28+24+12+55-9-5-2=103.
\end{eqnarray*}
\end{exam}
\end{proof}

\subsubsection{\protect\bigskip Binet Formula}

We have the following corollary by (1.5) and (Theorem 2.10 $iii$).

\begin{cor}
For $1\leq i\leq k$ and $m,n\in 
%TCIMACRO{\U{2124} }%
%BeginExpansion
\mathbb{Z}
%EndExpansion
^{+},$%
\begin{equation*}
l_{k,n}^{i}=\left\{ 
\begin{array}{l}
\sum\limits_{j=1}^{k}j\sum\limits_{i=1}^{k}\frac{(\lambda _{i})^{n-j}}{P%
%TCIMACRO{\U{b4}}%
%BeginExpansion
{\acute{}}%
%EndExpansion
(\lambda _{i})}\text{ \ \ \ \ \ \ \ \ \ \ \ \ \ \ \ \ \ \ for }i=1 \\ 
\sum\limits_{m=1}^{i}\sum\limits_{j=1}^{k}j\sum\limits_{i=1}^{k}\frac{%
(\lambda _{i})^{n-m-j+1}}{P%
%TCIMACRO{\U{b4}}%
%BeginExpansion
{\acute{}}%
%EndExpansion
(\lambda _{i})}\text{ for }1<i<k \\ 
\sum\limits_{j=1}^{k}j\sum\limits_{i=1}^{k}\frac{(\lambda _{i})^{n-j+1}}{P%
%TCIMACRO{\U{b4}}%
%BeginExpansion
{\acute{}}%
%EndExpansion
(\lambda _{i})}\text{ \ \ \ \ \ \ \ \ \ \ \ \ \ \ \ \ \ for }i=k%
\end{array}%
\right. .
\end{equation*}
where $l_{k,n}^{i}$ is the $k$SO$k$L.
\end{cor}

\bigskip We have the following corollary by (1.7) and (Theorem 2.10 $iii$).

\begin{cor}
Let $l_{k,n}^{i}$ be the $k$SO$k$L. Then, for $1\leq i\leq k$ and $m,n\in 
%TCIMACRO{\U{2124} }%
%BeginExpansion
\mathbb{Z}
%EndExpansion
^{+},$%
\begin{equation*}
l_{k,n}^{i}=\left\{ 
\begin{array}{l}
\sum\limits_{j=1}^{k}j\frac{\text{det}(V_{k,n-j}^{(1)})}{\text{det(}V\text{)}%
}\text{ \ \ \ \ \ \ \ \ \ \ \ \ \ \ \ \ \ \ for }i=1 \\ 
\sum\limits_{m=1}^{i}\sum\limits_{j=1}^{k}j\frac{\text{det}%
(V_{k,n-m-j+1}^{(1)})}{\text{det(}V\text{)}}\text{ \ for }1<i<k \\ 
\sum\limits_{j=1}^{k}j\frac{\text{det}(V_{k,n-j+1}^{(1)})}{\text{det(}V\text{%
)}}\text{ \ \ \ \ \ \ \ \ \ \ \ \ \ \ \ \ \ for }i=k%
\end{array}%
\right.
\end{equation*}%
where $V_{k,n-s}^{(1)}$ is a new notation for (1.6) which depends on n,
i.e., $V_{k,n-s}^{(1)}$ is a $k\times k$ matrix obtained from $V$ by
replacing $k$-th column of $V$ by%
\begin{equation*}
d_{k,n-s}^{(1)}=\left[ 
\begin{array}{c}
\lambda _{1}^{k-1+n-s} \\ 
\lambda _{2}^{k-1+n-s} \\ 
\vdots \\ 
\lambda _{k}^{k-1+n-s}%
\end{array}%
\right] .
\end{equation*}
\end{cor}

\subsection{\protect\bigskip Combinatorial Representation of the Generalized
Order-$k$ Fibonacci and Lucas Numbers}

In this subsection, we obtain some combinatorial representations of $i$-th
sequences of $k$SO$k$F and $k$SO$k$L with the help of combinatorial
representations of Generalized Fibonacci and Lucas Polynomials. \newline
$i$-th sequences of $k$SO$k$F can be stated in terms of $k$-th sequences of $%
k$SO$k$F as follows. For $c_{i}=1$ $(1<i<k),$ 
\begin{equation*}
f_{k,n}^{\text{ }i}=\sum\limits_{m=1}^{k-i+1}f_{k,n-m+1}^{k}.
\end{equation*}%
For $t_{i}=1$ $(1<i<k),$ $F_{k,n-1}(t)$ is reduced to sequence $f_{k,n}^{k}.$
So for $t_{i}=1$ $(1<i<k)$, $f_{k,n}^{\text{ }i}=\sum%
\limits_{m=1}^{k-i+1}F_{k,n-m}(t)$ and using (1.9) we have 
\begin{equation*}
f_{k,n}^{\text{ }i}=\sum\limits_{m=1}^{k-i+1}\sum\limits_{a\vdash \left(
n-m\right) }\binom{\left\vert a\right\vert }{a_{1,\ldots ,}a_{k}}.
\end{equation*}%
It is obvious that, for $t_{i}=1$ $(1<i<k)$, $F_{k,n}(t)=f_{k,n}^{1}$\ and $%
F_{k,n}(t)=f_{k,n+1}^{k},$ respectively. Then, for all $m,n\in 
%TCIMACRO{\U{2124} }%
%BeginExpansion
\mathbb{Z}
%EndExpansion
^{+},$ 
\begin{equation}
f_{k,n}^{\text{ }i}=\left\{ 
\begin{array}{c}
\sum\limits_{a\vdash n}\binom{\left\vert a\right\vert }{a_{1,\ldots ,}a_{k}}%
\text{ \ \ \ \ \ \ \ \ \ \ \ \ \ \ \ \ \ \ \ \ \ \ \ \ \ \ \ \ \ \ \ \ if }%
i=1 \\ 
\sum\limits_{m=1}^{k-i+1}\sum\limits_{a\vdash (n-m)}\binom{\left\vert
a\right\vert }{a_{1,\ldots ,}a_{k}}\text{ \ \ \ \ \ \ \ \ \ \ \ \ \ if }1<i<k
\\ 
\sum\limits_{a\vdash (n-1)}\binom{\left\vert a\right\vert }{a_{1,\ldots
,}a_{k}}\text{ \ \ \ \ \ \ \ \ \ \ \ \ \ \ \ \ \ \ \ \ \ \ \ \ \ }\ \ \ 
\text{if }i=k%
\end{array}%
\right. .
\end{equation}%
\ \ \ \ 

\begin{lem}
$\left[ 5\right] $ \bigskip Let $f_{k,n}^{k}$ be the $k$-th sequences of \ $%
k $SO$k$F, then,%
\begin{equation*}
f_{k,n}^{k}=\sum\limits_{m\vdash (n-1+k)}\frac{m_{k}}{\left\vert
m\right\vert }\times \binom{\left\vert m\right\vert }{m_{1},\ldots ,m_{k}}
\end{equation*}%
where $m=(m_{1},m_{2},\ldots ,m_{k})$ nonnegative integers satisfying 
\newline
$m_{1}+2m_{2}+\ldots +km_{k}=n-1+k.$ In addition for $0\leq i\leq n-1$%
\begin{equation*}
f_{k,n-i}^{k}=\sum\limits_{m\vdash (n-i+k-1)}\frac{m_{k}}{\left\vert
m\right\vert }\times \binom{\left\vert m\right\vert }{m_{1},\ldots ,m_{k}}
\end{equation*}%
where the summation is over nonnegative integers satisfying \newline
$m_{1}+2.m_{2}+\ldots +k.m_{k}=n-1-i+k.$
\end{lem}

\bigskip Then we have the following corollary using (Theorem 2.10. $iii$).

\begin{cor}
Let $l_{k,n}^{i}$ be the $k$SO$k$L, then, for $m,n\in 
%TCIMACRO{\U{2124} }%
%BeginExpansion
\mathbb{Z}
%EndExpansion
^{+},$%
\begin{equation*}
l_{k,n}^{i}=\left\{ 
\begin{array}{l}
\sum\limits_{i=1}^{k}j\sum\limits_{m\vdash (n-j+k-1)}\frac{m_{jk}}{%
\left\vert m\right\vert }\times \binom{\left\vert m\right\vert }{%
m_{j1},\ldots ,m_{jk}}\text{ \ \ \ \ \ \ \ \ \ if }i=1 \\ 
\sum\limits_{m=1}^{i}\sum\limits_{j=1}^{k}j\sum\limits_{t\vdash (n-m-j+k)}%
\frac{t_{mjk}}{\left\vert t\right\vert }\times \binom{\left\vert
t\right\vert }{t_{mj1},\ldots ,t_{mjk}}\ \ \text{ \ if }1<i<k \\ 
\sum\limits_{i=1}^{k}j\sum\limits_{m\vdash (n-j+k)}\frac{m_{jk}}{\left\vert
m\right\vert }\times \binom{\left\vert m\right\vert }{m_{j1},\ldots ,m_{jk}}%
\text{\ \ \ \ \ \ \ \ \ \ \ \ \ if }i=k%
\end{array}%
\right.
\end{equation*}%
\bigskip where $t=(t_{mj1},\ldots ,t_{mjk})$ and $m=(m_{j1},m_{j2},\ldots
,m_{jk}).$
\end{cor}

\begin{cor}
Let $l_{k,n}^{i}$ be the $k$SO$k$L, then, for all $m,n\in 
%TCIMACRO{\U{2124} }%
%BeginExpansion
\mathbb{Z}
%EndExpansion
^{+}$%
\begin{equation*}
l_{k,n}^{i}=\left\{ 
\begin{array}{c}
\sum\limits_{a\vdash (n-1)}\frac{n-1}{\mid a\mid }\binom{\left\vert
a\right\vert }{a_{1,\ldots ,}a_{k}}\text{ \ \ \ \ \ \ \ \ \ \ \ \ \ \ \ if }%
i=1 \\ 
\sum\limits_{m=1}^{i}\sum\limits_{a\vdash (n-m)}\frac{n-m}{\mid a\mid }%
\binom{\left\vert a\right\vert }{a_{1,\ldots ,}a_{k}}\text{ if }1<i<k \\ 
\sum\limits_{a\vdash n}\frac{n}{\mid a\mid }\binom{\left\vert a\right\vert }{%
a_{1,\ldots ,}a_{k}}\text{ \ \ \ \ \ \ \ \ \ \ \ \ \ \ \ \ \ \ if }i=k%
\end{array}%
\right. .
\end{equation*}

\begin{proof}
For $t_{i}=1(1\leq i\leq k),$ $G_{k,n}$ is reduced to $l_{k,n}^{k}.$ Since $%
l_{k,n}^{k}=\sum\limits_{a\vdash n}\frac{n}{\mid a\mid }\binom{\left\vert
a\right\vert }{a_{1,\ldots ,}a_{k}}$ from (1.10) and by using (2.7) the
proof is completed.
\end{proof}
\end{cor}

\begin{cor}
\bigskip Let $l_{k,n}^{i}$ be the $k$SO$k$L, then, for $1\leq i\leq k$ and $%
m,n\in 
%TCIMACRO{\U{2124} }%
%BeginExpansion
\mathbb{Z}
%EndExpansion
^{+}$%
\begin{equation*}
l_{k,n}^{i}=\left\{ 
\begin{array}{c}
\sum\limits_{j=1}^{k}j\sum\limits_{a\vdash (n-1-j)}\binom{\left\vert
a\right\vert }{a_{1,\ldots ,}a_{k}}\text{ \ \ \ \ \ \ \ \ \ \ \ \ if }i=1 \\ 
\sum\limits_{m=1}^{i}\sum\limits_{j=1}^{k}j\sum\limits_{a\vdash (n-m-j)}%
\binom{\left\vert a\right\vert }{a_{1,\ldots ,}a_{k}}\text{ \ if }1<i<k \\ 
\sum\limits_{j=1}^{k}j\sum\limits_{a\vdash (n-j)}\binom{\left\vert
a\right\vert }{a_{1,\ldots ,}a_{k}}\text{ \ \ \ \ \ \ \ \ \ \ \ \ \ \ \ \ \
if }i=k%
\end{array}%
\right. .
\end{equation*}
\end{cor}

\begin{proof}
\bigskip Proof is trivial from (1.9), (2.7).
\end{proof}

\begin{cor}
Let $l_{2,n}^{2}$ be the second sequence of the $2$SO$2$L, then,%
\begin{equation*}
l_{2,n}^{2}=\sum\limits_{j=1}^{2}j\sum\limits_{s=0}^{\left\lceil \frac{n-j}{2%
}\right\rceil }\binom{n-j-s}{s}
\end{equation*}%
where $\binom{n}{s}$ is combinations $s$ of $n$ objects, such that $\binom{n%
}{s}=0$ if $n<s.$
\end{cor}

\begin{proof}
In (1.11), $F_{2,n}(t)=\sum\limits_{j=0}^{\left\lceil \frac{n}{2}%
\right\rceil }(-1)^{j}\binom{n-j}{j}F_{1}^{n-2j}(t)(-t_{2})^{j}$ and for $%
t_{i}=1$ and $c_{i}=1$ \newline
$(1\leq i\leq k),\ F_{2,n-1}(t)$ is reduced to sequence $f_{k,n}^{2}$. Proof
is completed by using $f_{k,n}^{2}$($c_{i}=1$ for $1\leq i\leq k$) and
(2.10. $iii$).
\end{proof}

\proof[Acknowledgements] The authors are grateful for financial support from
the Research Fund of Gaziosmanpa\c{s}a University under grand no:2009/46.

\bigskip

\bigskip

\end{document}